\documentclass[12pt]{amsart}
\usepackage{adjustbox}   
\usepackage{amsmath, amssymb,mathtools,amsthm,amscd,
multirow, tikz-cd,blindtext, graphicx, float, tabularx, lipsum}
\usepackage[normalem]{ulem}
\usepackage[all]{xy}
\usepackage[english]{babel}
\definecolor{dodgerblue}{rgb}{0.12, 0.56, 1.0}
\usepackage[raiselinks=false,colorlinks=true,citecolor=blue,urlcolor=dodger blue,
linkcolor=blue,bookmarksopen=true,pdftex]{hyperref}

\tikzcdset{scale cd/.style={every label/.append style={scale=#1},cells={nodes={scale=#1}}}}
\usepackage[mathscr]{eucal}
\newtheorem{theorem}{Theorem}[section]
\newtheorem{proposition}[theorem]{Proposition}
\newtheorem{lemma}[theorem]{Lemma}
\newtheorem{remark}[theorem]{Remark}

\newtheorem{corollary}[theorem]{Corollary}

\begin{document}

\newcommand{\ind}{{\rm ind}}
\newcommand{\lra}[1]{ \langle #1 \rangle}
\def\inn{\mathrm{Inn}}
\def\myhom{\mathrm{Hom}}
\def\gcd{\mathrm{gcd}}
\newcommand{\tmod}[1]{\; (#1)}
\newcommand{\myendo}[2]{{\rm End}(#1,#2)}
\newcommand\citeY[1]{\citeauthor{#1} (\citeyear{#1})}
\newcommand\citeN[1]{[\citenumber{#1}]}

\newcommand{\af}{\alpha}
\newcommand{\et}{\eta}
\newcommand{\ga}{\gamma}
\newcommand{\Ga}{\Gamma}
\newcommand{\ta}{\tau}
\newcommand{\ph}{\varphi}
\newcommand{\bt}{\beta}
\newcommand{\lb}{\lambda}
\newcommand{\Lb}{\Lambda}
\newcommand{\wh}{\widehat}
\newcommand{\sg}{\sigma}
\newcommand{\Sg}{\Sigma}
\newcommand{\om}{\omega}
\newcommand{\Om}{\Omega}

\newcommand{\bb}{\mathbb}
\newcommand{\mathbbN}{\mathbb N}
\newcommand{\mathbbZ}{\mathbb{Z}}
\newcommand{\mathbbR}{\mathbb R}
\newcommand{\mathbbC}{\mathbb C}

\newcommand{\cH}{\mathcal H}
\newcommand{\cF}{\mathcal F}
\newcommand{\cN}{\mathcal N}
\newcommand{\cR}{\mathcal R}
\newcommand{\cc}{\mathcal C}
\newcommand{\cP}{\mathcal P}
\newcommand{\cW}{\mathcal W}

\newcommand{\T}{\mathrm{T}}
\newcommand{\im}{\mathrm{Im}}
\newcommand{\SO}{\mathrm{SO}}
\renewcommand{\O}{\mathrm{O}}
\newcommand{\Spin}{\mathrm{Spin}}
\newcommand{\E}{\mathrm E}
\newcommand{\tor}{\mathrm{Tor}}
\newcommand{\rank}{\mathrm{rank}}
\newcommand{\Mod}[1]{\ (\mathrm{mod}\ #1)}
\newcommand{\GL}{\mathrm{GL}}
\newcommand{\SU}{\mathrm{SU}}
\newcommand{\dprime}{{\prime\prime}}
\newcommand{\p}{\prime}
\newcommand{\rH}{\mathrm H}

\newcommand{\bea} {\begin{eqnarray*}}
\newcommand{\beq} {\begin{equation}}
\newcommand{\bey} {\begin{eqnarray}}
\newcommand{\eea} {\end{eqnarray*}}
\newcommand{\eeq} {\end{equation}}
\newcommand{\eey} {\end{eqnarray}}
\newcommand{\ovl}{\overline}
\newcommand{\I}{{\imath}}

\newcolumntype{P}[1]{>{\centering\arraybackslash}p{#1}}

\title{The $\cP$-triviality of Stunted Projective Spaces}
\author[S. Podder]{Sudeep Podder}
\address{Indian Statistical Institute, 8th Mile, Mysore Road, RVCE Post, Bangalore 560059, India}
\email{sudeeppodder1993@gmail.com, sudeep\_pd@isibang.ac.in}
\author[G. Sau]{Gobinda Sau}
\address{Indian Statistical Institute, 8th Mile, Mysore Road, RVCE Post, Bangalore 560059, India}
\email{gobindasauanalytica@gmail.com, gobindasau@isibang.ac.in}
\begin{abstract}
In this article, we introduce the notion of $\cP$-triviality of topological manifolds and give a complete description of the $\cP$-triviality of stunted real and complex projective spaces.  
\end{abstract}

\maketitle
\section{Introduction}\label{intro}
The vanishing of characteristic classes of vector bundles over a fixed space has been recently studied by various authors; see, for instance, \cite{Tanaka2010}, \cite{NT2014}, \cite{BKN2024}.  It is well known due to Atiyah and Hirzebruch
(see \cite[Theorem 2]{AH1961}), that the 9-fold suspension of any finite CW complex is $\cW$-trivial. A space $X$ is called $\cW$-trivial if the total Stiefel-Whitney class $w(\alpha)$ is 1 for every vector bundle $\alpha$ over $X$. Such a problem of finding trivialities of characteristic classes was also introduced by Tanaka in \cite{Tanaka2010} due to its relation with certain Borsuk-Ulam type theorems for every vector bundle over a fixed base space. He proved that the $8$-fold suspension $\Sigma^8X$ of a space $X$ is $\cW$-trivial if either $X$ is $\cW$-trivial.
or the cup product in $\widetilde{H}^{*}(X;\mathbb{Z}_2)$ is trivial. Note that if $X$ is a suspension, the cup product in $\widetilde{H}^{*}(X;\mathbb{Z}_2)$ is trivial and the Atiyah-Hirzebruch theorem will be obtained by applying Tanaka's result to $X$. 
 
Let $X$ be CW complex. We define $X$ to be $\cP$-trivial if all the \textit{Pontrjagin classes} for any vector bundle $\xi$ over $X$ vanish, that is, the \textit{total Pontrjagin class} $p(\xi)=1$ for every vector bundle $\xi$ over $X$ (see \S2). In this article, we study the notion of $\cP$-triviality and determine whether some of the familiar spaces and their suspensions are $\cP$-trivial.\\

The real projective space $\bb{RP}^n$ is the $n$-th skeleton of $\bb{RP}^m$ for $m\ge n$. The quotient space $\bb{RP}^m/\bb{RP}^n$ is called a stunted real projective space, and let us denote it by $X_{m, n}$. We write $X_m$ for $\bb{RP}^m$ and may call it a stunted space by identifying it with $X_{m, 0}$. James introduced these spaces in \cite{James1959} to compute sections of the Stiefel Manifolds. Adams, in \cite{Adams1962}, proving specific \textit{co-reducibility} results of these spaces, settled the classical vector field problem on spheres. It is interesting to compute the behaviour of characteristic classes of vector bundles over the stunted projective spaces. 

The trivialities of Chern and Stiefel-Whitney classes of vector bundles over the iterated suspension of the real projective spaces were extensively discussed in \cite{NT2014}, \cite{NT2014(1)}, \cite{BKN2024}. For a comprehensive understanding of certain characteristic classes of vector bundles over the iterated suspension of spaces, we refer to Tanaka \cite{Tanaka2010}. We denote the $k$-th iterated suspension of $X_{m, n}$ by $X^k_{m, n}$. We have the following triviality result of the Pontrjagin classes of vector bundles over $X^k_{m, n}$. Let us define, as in \cite{Adams1962}, the function $\varphi$ such that $\varphi(m, n)$ to be the number of integers $n<s\leq m$ such that $s\equiv 0, 1, 2$ or $4 \pmod 8$.
\begin{theorem}\label{main1}
    $X_{m,n}^k$ is not $\cP$-trivial if and only if one of the following holds.\\
    (a) $k=0=n$ and $m\ge 4$,\\
    (b) $n=0$ and $k+m\equiv 0\pmod 4$,\\
    (c) Let $k=0$ and $g = \varphi(n, 0)$. Either $n\equiv 3\pmod 4$ or $m\ge 2^{g+1}$.\\
    (d) $k>0$, $k$ is odd and $(k+m)\equiv 0\pmod 4$, or $k$ is even and $k+n+1\equiv 0\pmod 4.$
\end{theorem}
One can define the complex stunted projective spaces by replacing the real projective spaces with complex ones. We denote these spaces by $Y_{m, n}$, and $Y^k_{m, n}$ analogously. We prove the following triviality result of the Pontrjagin classes of vector bundles over $Y^k_{m, n}$.
\begin{theorem}\label{main2}
    $Y^k_{m, n}$ is $\cP$-trivial if and only if one of the followings hold.\\
    (a) $k$ is odd,\\
    (b) $n = m-1(\neq 0)$ and $(2m+k)\not\equiv 0\pmod 4$.
\end{theorem}
 In this article, the word suspension should always mean reduced suspension. While the primary focus is on $\cP$-triviality of stunted real and complex projective spaces, we also proved some results about $\cP$-triviality of certain CW complexes, which will be interesting in other contexts. Using Cadek and Vanzura's result \cite[Theorem 1, p. 757]{cadek1993}, we proved that the third suspension of the compact orientable surface of genus $\geq 2$ is never $\cP$-trivial. We have also dealt with $\cP$-triviality results of the product of certain manifolds.

The article is organized as follows. In Section \ref{2nd}, we discuss some preliminaries and some general observations in Section \ref{3rd}. The proof of the main theorems is given in Sections \ref{4th} and \ref{5th}, respectively.
\section{Preliminaries}\label{2nd}
Let $\xi$ be a real vector bundle over a topological space $X$. We denote the complexification $\xi\otimes\bb C$ of $\xi$ by $\xi_\bb C$, whose fibres $F_x(\xi_\bb C)$ are the complexification $F_x(\xi)\otimes \bb C$ of fibres $F_x(\xi)$ of $\xi$. For $0\leq j\le \lfloor \mathrm{dim}(X)/4 \rfloor$, the $j$-th Pontrjagin class $p_j(\xi)$ of $\xi$ is defined to be $(-1)^jc_{2j}(\xi_\bb C)$ where $c_{2j}(\xi_\bb C)$ is the $2j$-th Chern class of $\xi_\bb C$. The total Pontrjagin class $p(\xi)$ of $\xi$ is defined to be  $\sum_{j\ge 0}p_j(\xi)$. We call $X$ to be $\cP$-trivial if $p(\xi)=1$ for any real vector bundle $\xi$ over $X$. \\

If $\xi$ is a complex vector bundle and $\xi_\bb R$ denotes the underlying real vector bundle of $\xi$, then $\xi_\bb R\otimes \bb C$ is canonically isomorphic to $\xi\oplus \bar{\xi}.$(See \cite[Lemma 15.4]{milnor1974}.) Let us denote the total Chern class $\sum_{j\ge 0} c_j(\xi)$ by $c(\xi)$. It is easy to see using the formula for the total Chern class for Whitney sum that 
\begin{equation}\label{chern_pontrjagin}
    p_j(\xi_\bb R)=c_{j}^2(\xi)-2c_{j+1}(\xi)c_{j-1}(\xi)+\cdots\pm 2c_{1}(\xi)c_{2j-1}(\xi)\mp 2c_{2j}(\xi).
\end{equation}
It is worth noting that since $\xi_\bb C$ is isomorphic to $\bar{\xi}_\bb C$ for a real vector bundle $\xi$, the odd-dimensional Chern classes of $\xi_\bb C$ are order two elements of the integral cohomology. Since the Pontrjagin classes are integral cohomology classes of dimension multiple of four, it is important to know the cohomology of the spaces under consideration. We give a brief description of the cohomology groups of $X^k_{m,n}$ and of $Y^k_{m,n}$ in this section.\\

Recall that the cellular chain complex for the integral homology of $\bb{RP}^m$ is given by
\[
\cdots\to 0\to \bb Z\xrightarrow{d_m} \bb Z\xrightarrow{d_{m-1}} \cdots \xrightarrow{d_0} \bb Z,
\]
where the degrees of the boundary maps are given by $deg(d_j)= 1+(-1)^j, 0\leq j\leq n$. Since, for $n\le m, \bb{RP}^n$ is the $n$-th skeleton of $\bb{RP}^m$, the stunted projective space $X_{m,n}$ has one cell $\bb{RP}^j$ in each dimension $n+1\le j\le m$ and in dimension 0. the cellular chain complex for $X_{m,n}$ is then given by  
\begin{equation}\label{stunted_chain_complex}
    \cdots\to 0\to \bb Z\xrightarrow{d_m} \bb Z\xrightarrow{d_{m-1}} \cdots \xrightarrow{d_{n+2}}\bb Z\xrightarrow{d_{n+1}} 0\to\cdots\to 0 \xrightarrow{d_0} \bb Z.
\end{equation}
If $n$ is odd, $d_{n+2} =0$ and if $n$ is even, then $d_{n+2}$ is the multiplication by 2 homomorphism. We have the following description of the integral homology of $X_{m,n}$.
\begin{equation}\label{stunted_homology}
    H_j(X_{m,n}; \bb Z)=\begin{cases}
                          \mathbb{Z} &\textit{if $j=0, j=n+1$ even, $j=m$ odd,}\\
                          \mathbb{Z}_2 &\textit{if $j$ is odd, and $n<j< m$,}\\
                          0 & \textit{otherwise.}
                      \end{cases}
\end{equation}
Since $X_{m,n}$ is path connected, we have $H_j(\Sigma^kX_{m,n}; \bb Z)\cong H_{j-k}(X_{m,n}; \bb Z)$ and hence,
\begin{equation}\label{suspension_stunted_homology}
   H_j(X^k_{m,n}; \bb Z)=\begin{cases}
                            \mathbb{Z} &\textit{if $j=0, j-k=n+1$ even, $j-k=m$ odd,}\\
                            \mathbb{Z}_2 &\textit{if $j-k$ is odd, and $n+k<j< m+k$,}\\
                            0 & \textit{otherwise.}
                        \end{cases} 
\end{equation}
Using the universal coefficient theorem, we have the following description of the integral cohomology of $X^k_{m, n}$.
\begin{equation}\label{suspension_stunted_cohomology}
    H^j(X^k_{m,n}; \bb Z)=\begin{cases}
                            \mathbb{Z} &\textit{if $j=0, j-k=n+1$ even, $j-k=m$ odd,}\\
                            \mathbb{Z}_2 &\textit{if $j-k$ is even, and $n+k+1<j\le m+k$,}\\
                            0 & \textit{otherwise.}
                        \end{cases}
\end{equation}
Note that for $n=0=k$ this matches with the usual description of the integral cohomology of $\bb {RP}^n.$

Let us describe the integral cohomology groups of $Y^k_{m,n}$. Recall that $\bb{CP}^i$ is the $2i$-th skeleton of $\bb {CP}^m$ for $0\le i\le m$, and there is one cell in each even dimension. The cellular boundary maps are all trivial, and the chain groups are, therefore, isomorphic to the integral homology groups. Since all the boundary maps are already trivial, collapsing the $n$-th skeleton, $\bb {CP}^n$, does not effect the cohomology groups except for all non-zero dimensions less than $n+1$, where all the cohomology groups become trivial. Hence we have,

\begin{equation}\label{complex_stunted_cohomology}
    H^j(Y_{m, n}; \bb Z)=\begin{cases}
                            \bb Z  &\textit{if $j=0$ or $2n+2\le j\le 2m$ is even,}\\
                            0      & \textit{otherwise.} 
                       \end{cases}
\end{equation}
                       
Therefore, as $Y^k_{m,n}$ is path connected, we have the following description of the integral cohomology of $Y^k_{m,n}$.

\begin{equation}\label{suspension _complex_stunted_cohomology}
    H^j(Y^k_{m, n}; \bb Z)=\begin{cases}
                            \bb Z  &\textit{if $j=0$ or $2n+k+2\le j\le 2m+k, j-k$ is even,}\\
                            0      & \textit{otherwise.} 
                         \end{cases}
\end{equation}

We often use the following comparison lemma to conclude the $\cP$-triviality. Recall the reduced $KO$ group $\widetilde{KO}(X)$ of a space $X$. Any continuous map $f\colon X\to Y$ to another space $Y$ induces the homomorphism $f^*\colon \widetilde{KO(Y)}\to \widetilde{KO(X)}$ that maps an isomorphism class of (real) vector bundle to the isomorphism class of its pullback bundle. We shall denote the homomorphisms induced on the reduced complex $K$-theory and (singular) cohomology by $f^*$ as well, and the notation should always be clear from the context. We refer to \cite{Fujii1967} for the $KO$ groups of stunted real projective spaces.

\begin{lemma}\label{comparison_lemma}
    Let $Y$ be a space with a map $f\colon Y\to X$, which induces a surjective homomorphism on the reduced real or complex $K$-groups. Then, if $Y$ be $\cP$-trivial, so is $X$.
\end{lemma}

\begin{proof}
    We shall prove the real case only; the complex case follows parallelly. If possible, let $X$ be not $\cP$-trivial, and $\xi$ be a vector bundle over $X$ with at least one nontrivial Pontrjagin class $p_j(\xi)$, say. Since $f^*$ is surjective, there is a bundle $\eta$ over $Y$ such that $f^*([\eta])=[\xi]$. Now, as $f^*(p_j(\eta))=p_j(\xi)\neq 0$, we must have $p_j(\eta)\neq 0$. This contradicts that $Y$ is $\cP$-trivial. Hence $X$ must be $\cP$-trivial.
\end{proof}

It is evident, from the definition of Pontrjagin classes, that a space which is $\cc$-trivial is always $\cP$-trivial. We shall recall the following relations between different characteristic classes, which we will use in the following sections.

Let us consider the reduction modulo 2 homomorphism $\rho_2\colon H^j(X; \bb Z)\to H^j(X, \bb Z_2), 0\le j\le \mathrm{dim}(X)$ induced by the quotient map from $\bb Z$ to $\bb Z_2$. Since $\rho_2$ maps the total Chern class to the total Stiefel-Whitney class, it follows that for $0\le j\le \lfloor \mathrm{dim}(X)/4 \rfloor$, we have (see \cite[15-A]{milnor1974})
\begin{equation}\label{Pontrjagin_Whitney}
    \rho_2(p_j(\xi))=w_{2j}^2(\xi).
\end{equation}
Let $\rho_4\colon H^j(X; \bb Z)\to H^j(X; \bb Z_4)$ be the reduction modulo four map and $i_*\colon H^j(X; \bb Z_2)\to H^j(X; \bb Z_4)$ be the homomorphism induced by the inclusion $i\colon \bb Z_2\to \bb Z_4$. Let $\mathfrak{P}:H^{2j}(Y;\mathbb{Z}_2)\rightarrow H^{4j}(Y;\mathbb{Z}_4)$ denote the Pontrjagin square operation. Then the Pontrjagin class $p_j(\xi)$ is related to the Stiefel-Whitney classes $w_1, \ldots, w_{4j}$ as follows (see \cite[Lemma 1]{Milnor1958}).
\[
\mathfrak{P}(w_{2j}(\xi)) = \rho_4(p_j(\xi))+i_*(f_j(w_1(\xi), \ldots, w_{4j}(\xi)),
\]
where $f_j$ is a certain mod 2 polynomials. In particular, for $j=1, f_j(w_1(\xi), \ldots, w_{4}(\xi))=w_1(\xi)Sq^1(w_2(\xi))+w_0(\xi)w_4(\xi),$ and, if $\xi$ is orientable or $w_2(\xi)=0$, we have (see \cite{cadek1993})
\begin{equation}\label{Pontrjagin_Square}
    \mathfrak{P}(w_{2}(\xi)) = \rho_4(p_1(\xi))+i_*( w_{4}(\xi)).
\end{equation}

\section{Generalities}\label{3rd}
In this section, we derive general observations and the relation between the trivialities of different characteristic classes. 
Let $a\in H^{4n}(S^{4n}; \bb Z)$ be a generator. Recall that for each $m$ divisible by $(2n-1)!$, there exists a unique $\xi\in \widetilde{K}(S^{4n})$ with $c_{2n}(\xi)=ma$. (See \cite[Chapter 20, Corollary 9.8]{hu}.) From (\ref{chern_pontrjagin}), we have $p_n(\xi_\bb R)=\pm 2c_{2n}(\xi)\neq 0$. Therefore $S^n$ is not $\cP$-trivial if $m\equiv 0\pmod 4$. Since $S^n$ has nontrivial cohomology groups in dimension 0 and $n$ only and Pontrjagin classes are $4j$-dimensional cohomology classes, $0\le j\le \lfloor n/4\rfloor$ , we have,
\begin{lemma}\label{spheres}
$S^n$ is not $\cP$-trivial if and only if $n\equiv 0\pmod 4$.$\hfill\square$ 
\end{lemma}
Therefore, $\cW$-triviality of a space does not imply neither $\cc$-triviality, nor $\cP$-triviality, in fact, for $n>2, S^{4n}$ is $\cW$-trivial but neither $\cc$-triviality, nor $\cP$-triviality. Since $\Sigma^k S^n\simeq S^{n+k}$, the same example shows that no analogue of Atiyah-Hirzebruch holds for $\cP$-triviality. If, however, $\rho_2$ is injective whenever $H^{4j}(X; \bb Z)\neq 0$, it is clear from \eqref{Pontrjagin_Whitney} that $X$ is $\cP$-trivial. 

Recall that the cup length $\mathrm{cup}(X)$ of a cohomology ring $H^{*}(X; R)$ is the largest $k$ such that there exists $a_1.a_2,\cdots,a_k \in H^{*}(X;R)$ with $a_1 \smile a_2 \smile \cdots \smile a_k\neq 0$. For any vector bundle $\xi$ over $X$, if one of the three conditions in the next lemma holds, we have $\rho_2(p(\xi))=0$ from \eqref{Pontrjagin_Whitney}. Hence, we have,

\begin{lemma}\label{suspension}
Let $\rho_2$ is be injective whenever $H^{4j}(X; \mathbb Z)$ is non-trivial, $ 0\le j\le \lfloor \mathrm{dim}(X)/4\rfloor$. Then $X$ is $\mathcal P$-trivial if one of the followings hold.\\
(a) $\mathrm{cup}(X)=1$,\\
(b) $w_{2j}^2(\xi)=0$ for all $j$ and all bundles $\xi$ over $X$,\\
(c) $Sq^{2j}\colon H^{2j}(X; \bb Z_2)\to H^{4j}(X; \bb Z_2)$ is the zero homomorphism for all $j$.$\hfill\square$   
\end{lemma}

If $X = \Sigma Y$ for some space $Y$, it is well-known that $\mathrm{cup}(X)=1$. Therefore, we have, 
\begin{corollary}\label{rho2_injective}
If $\rho_2\colon H^{4j}(X; \mathbb Z)\to H^{4j}(X; \mathbb Z_2)$ is injective whenever $H^{4j}(X; \mathbb Z)\neq 0$, and $X$ is a suspension, then $X$ is $\mathcal P$-trivial.\hfill$\square$
\end{corollary}
\begin{lemma}\label{square_of_Whitney_class}
For any real vector bundle $\xi$ over a $\mathcal P$-trivial space must satisfy $w_{2j}^2(\xi)=0$.
\end{lemma}
\begin{proof}
    The proof immediately follows from (\ref{Pontrjagin_Whitney}).
\end{proof}
\begin{corollary}\label{triviality_of_RPn}
    Since, for $m\geq 4$, the square of the second Stiefel-Whitney class of the tautological bundle $\gamma_m$ over $\mathbb {RP}^m$ is $w_4(\gamma_m)$, which is nonzero. Therefore, $\mathbb {RP}^m$ is not $\mathcal P$-trivial by the previous lemma.
\end{corollary}
Recall that if $M$ is a closed orientable manifold of dimension $n$, there is a degree one map from $M$ to the $n$ dimensional sphere $S^n$, inducing isomorphism between the top dimensional integral cohomology groups.
\begin{proposition}\label{suspension_of_orientable_manifolds}
Let $n, k\equiv 0\pmod 4$. If $M$ be a closed orientable manifold of dimension $n$, then $\Sigma^kM$ is not $\cP$-trivial. In particular, any closed orientable manifold of dimension $n\equiv 0\pmod 4$ is not $\cP$-trivial.
\end{proposition}
\begin{proof}
Consider a degree one map $f:M\rightarrow S^n.$ $f$ induces the map $\Sigma^kf\colon \Sigma^kM\to S^{k+n}$ such that $(\Sigma^k f)^*\colon H^{n+k}(\Sigma^kM; \bb Z)\to H^{n+k}(S^{k+n}; \bb Z)$ is an isomorphism. Since $n+k\equiv 0\pmod 4$ there is a bundle $\xi$ over $S^{n+k}$ such that $p_{\frac{n+k}{4}}(\xi)\neq 0$. (See the discussion before Lemma \ref{spheres}.) Then $p_{\frac{n+k}{4}}
((\Sigma^k f)^{*}\xi)\neq 0$, and hence, $\Sigma^kM$ is not $\cP$-trivial.
\end{proof}

It is easy to see that a $\cP$-trivial space may not be $\cc$-trivial in general. Indeed $S^n$, where $n\equiv 2\pmod 4$, is $\cP$-trivial but not $\cc$-trivial. Let $\mathfrak c\colon KO(X)\to K(X)$ be the complexification homomorphism induced by the map sending a real vector bundle to its complexification. The following proposition relates $\cc$-triviality to $\cP$-triviality.
\begin{proposition}
    Let $\mathfrak c$ be surjective, and the even-dimensional integral cohomology groups of $X$ have no 2-torsion. If $X$ is $\cP$ trivial, it is $\cc$-trivial.
\end{proposition}
\begin{proof}
    Let $\eta$ be a complex vector bundle over $X$. Then, as $\mathfrak{c}$ is surjective, $\eta=\xi_\bb C$ for some real vector bundle $\xi$ over $X$. Therefore, $c_{2j}(\eta)=p_j(\xi)=0$ as $X$ is $\cP$-trivial. Therefore, for any complex vector bundle $\eta$, all its even-dimensional Chern classes vanish. \\
    Now, recall that for a real bundle $\xi$ over $X$, the odd-dimensional Chern classes of $\xi_\bb C$ are order two elements of the integral cohomology of $X$. Since the integral cohomology groups have no $2$-torsion, the odd-dimensional Chern classes of $\eta$ are also zero. Hence $X$ is $\cc$-trivial.
\end{proof}
Let us now observe some consequences of a space being $\cP$-trivial in terms of the following two lemmas, which will be useful to determine when a space is not $\cP$-trivial.
\begin{lemma}\label{odd1}
Let $X$ be $\cP$-trivial. Then, $x^2=0$ for all $x\in H^2(X;\mathbb{Z})$.
\end{lemma}
\begin{proof}
If  $x\in H^2(X;\mathbb{Z})$, then there exists a line bundle $\xi$ such that $c_1(\xi)=x$. Form (\ref{chern_pontrjagin}) we have,
\[
p_1(\xi_{\mathbb{R}})=-2c_2(\xi)+c_1^2(\xi).
\]
Since $\xi$ is a line bundle, $c_2(\xi)=0$ yielding $p_1(\xi_{\mathbb{R}})=c_1^2(\xi)=x^2$. Since $X$ is $\cP$-trivial, $x^2=p_1(\xi_\bb R)=0$.
\end{proof}
\begin{lemma}
Let $X$ be $\cP$-trivial and $\xi$ be a complex $n$-plane bundle over $X$. Then $c_n^2(\xi)=0$.
\end{lemma}
\begin{proof}
If $\xi$ is a complex $n$-plane bundle, then $\xi_\bb R$ is an oriented $2n$-plane bundle over X. By \cite[Corollary 15.8]{milnor1974}, the square of the Euler class of $\xi_\bb R$ is the top dimensional Pontrjagin class, $i.e., p_n(\xi_{\mathbb{R}})=e(\xi_{\mathbb{R}})^2$. Again, the top dimensional Chern class of $\xi$ is, by definition, the Euler class of $\xi_\bb R$. Therefore, 
\[
c_n^2(\xi)=e(\xi_{\mathbb{R}})^2=p_n(\xi_{\mathbb{R}})=0.
\]
\end{proof}

Let $X, Y$ be compact CW complexes. If either $X$ or $Y$ is not $\cP$-trivial, so is $X\times Y$. In fact, the pullback of a bundle with a nontrivial Pontrjagin class under one of the projection maps to $X\times Y$ has a nontrivial Pontrjagin class.
This leads to the following lemma:
\begin{lemma}
 $S^n \times S^k$ is not $\cP$-trivial if and only if either $n+k\equiv 0\pmod 4$ or one of $n,k$ is $\equiv0 \pmod 4$.
\end{lemma}
\begin{proof}
Let $n+k\equiv 0\pmod 4$. Then, by Proposition \ref{suspension_of_orientable_manifolds}, we conclude that  $S^n \times S^k$ is not $\cP$-trivial. If one of $n,k$ is $\equiv0 \pmod 4$, the above discussion yields that  $S^n \times S^k$ is not $\cP$-trivial.

Conversely, assume that $S^n \times S^k$ is not $\cP$-trivial and $n+k \not \equiv 0 \pmod 4$. If possible, let $n\not\equiv 0\pmod 4$ and $k\not\equiv 0\pmod 4$. Since $n+k$ is not divisible by $4$, 
the pair $(n,k)$ is not of the form $(4p+1,4q+3),(4p+2,4q+2),(4p+3,4q+1)$. For any other pair $(n,k)$, it is clear that $H^{4i}(S^n \times S^k)=0$, for all $i$, yielding $\cP$-triviality of $S^n \times S^k$. Therefore, one of $n$ or $k$ must be divisible by $4$. This completes the proof.
\end{proof}
If $X$ and $Y$ are two CW complexes which are $\cP$-trivial, on the other hand, $X\times Y$ may not be $\cP$-trivial. For example, $S^2\times S^2$ is not $\cP$-trivial. The following lemma gives a sufficient condition for $\cP$-triviality of a product of spaces.
\begin{lemma}
Let $X,Y$ be two closed smooth manifolds with $\widetilde{KO}(X)=0$
and $\widetilde{KO}(Y)=0$. If $X\wedge Y$ is $\cP$-trivial, then the product $X\times Y$ is $\cP$-trivial.
\end{lemma}
\begin{proof}
 Let $\alpha$ be a real vector bundle over $X\times Y$. The co-fiber sequence
\[
X\vee Y\rightarrow X\times Y{\overset{q}\rightarrow} X\wedge Y
\]
gives the exact sequence
\[
\widetilde{KO}(X\wedge Y){\overset{q^{*}}\rightarrow}
\widetilde{KO}(X \times Y)\rightarrow \widetilde{KO}(X \vee Y).
\]
Since $\widetilde{KO}(X\vee Y)=\widetilde{KO}(X)\oplus\widetilde{KO}(Y)$, the last group is trivial, and hence, $q^{*}$ is surjective. There exists $\beta$ over $X\wedge Y$ such that $q^{*}(\beta)=\alpha$. By Lemma \ref{comparison_lemma}, since $X\wedge Y$ is $\cP$-trivial, $X\times Y$ is $\cP$-trivial. This completes the proof.
\end{proof}

We end this section with some general observations that are interesting in their own right.
\begin{proposition}
 Let $X$ be a compact orientable surface of genus $g\geq 1$. Then $\Sigma^3X$ is not $\cP$-trivial.   
\end{proposition}
\begin{proof}
Let $Y=\Sigma^3X$. Then, from the standard description of the cohomology groups of a surface of genus $g\ge 1$, we have $H^3(Y;\mathbb{Z}_2)=0, H^4(Y;\mathbb{Z}_4)=\mathbb{Z}_4^{2g}$, and $H^5(Y;\mathbb{Z})=\mathbb{Z}$.
Consider the Bockstein long exact sequences 
\[
\cdots \rightarrow H^3(Y;\mathbb{Z}_2)\rightarrow H^4(Y;\mathbb{Z}_2)\xrightarrow{i_{*}} H^4(Y;\mathbb{Z}_4)\rightarrow\cdots,
\] 
corresponding to the short exact sequence $\mathbb{Z}_2\xhookrightarrow{i} \mathbb{Z}_4\rightarrow \mathbb{Z}_2$, and 
\[
\cdots \rightarrow H^4(Y;\mathbb{Z})\xrightarrow{\rho_4} H^4(Y;\mathbb{Z}_4)\xrightarrow{\beta}H^5(Y;\mathbb{Z})\rightarrow\cdots,
\] 
corresponding to the short exact sequence $\mathbb{Z}\rightarrow \mathbb{Z}\rightarrow \mathbb{Z}_4$. Since $H^3(Y;\mathbb{Z}_2)=0, i_{*}$ is injective. And, since $ H^4(Y;\mathbb{Z}_4)=\mathbb{Z}_4^{2g}$ and $H^5(Y;\mathbb{Z})=\mathbb{Z}$, the map $\beta$ is identically zero and $\rho_4$ is surjective. Consequently, there exists $0\neq c\in H^4(Y;\mathbb{Z})$ and $b\in H^4(Y;\mathbb{Z}_2)$ such that 
\[
\rho_4(c)=i_{*}(b).
\] 
Since $H^2(Y;\mathbb{Z}_2)=0$, the Pontrjagin square $\mathfrak{P}:H^2(Y;\mathbb{Z}_2)\rightarrow H^4(Y;\mathbb{Z}_4)$ is zero. Therefore, from (\ref{Pontrjagin_Square}), we have, 
\[
\rho_4(c)=\mathfrak{P}(0)+i_{*}(b).
\]
Using Cadek and Vanzura's result \cite[Theorem 1, p. 757]{cadek1993}, there exists a bundle $\xi$ over $Y$ such that $p_1(\xi)\neq 0$ and consequently, $\Sigma^3X$ is not $\cP$-trivial.
\end{proof}
\begin{remark}
Note that, for a compact orientable surface $X$, the first suspension $\Sigma X$, being a $3$-dimensional CW complex, is $\cP$-trivial. On the other hand, $\Sigma^2 X$ is not $\cP$-trivial since there is a degree one map from $\Sigma^2 X$ to $\mathbb{S}^4$.
\end{remark}
\begin{proposition}\label{gen}
Let $Y$ be a $4$-dimensional path connected CW complex such that $\pi_1(Y)$ is perfect. Then $\Sigma^{4k-1}Y$ is $\cP$-trivial for all $k\in \mathbb{N}$.
\end{proposition}
\begin{proof}
Recall that a group $G$ is perfect if $[G, G]=G$. Since $\pi_1(Y)$ is perfect, we have $H_1(Y; \bb Z)=\pi_1(Y)/[\pi_1(Y), \pi_1(Y)]=0$. We show that all the cohomology groups in dimensions multiple of four vanishes.\\
 For $j<k$, since $Y$ is path connected, we have $H_{4j}(\Sigma^{4k-1}Y; \bb Z)\cong \widetilde{H}_0(\Sigma^{4k-4j+1}Y; \bb Z)=0$, and hence, $H^{4j}(\Sigma^{4k-1}Y; \bb Z)=0$.\\
 For $j=k$, since $H_{4k}(\Sigma^{4k-1}Y; \bb Z)\cong H_1(Y; \bb Z)=0$, and $H_{4k-1}(\Sigma^{4k-1}Y; \bb Z)\cong\widetilde{H}_0(Y; \bb Z)=0$, using Universal coefficient theorem for cohomology, we have $H^{4k}(\Sigma^{4k-1}Y; \bb Z)=0.$\\
 For $j>k$, since $4j> 4k+3=dim(\Sigma^{4k-1}Y)$, obviously $H^{4j}(\Sigma^{4k-1}Y; \bb Z)=0$.\\
 Therefore, all the cohomology groups in dimensions multiple of four vanish and $Y$ must be $\cP$-trivial.
\end{proof}

\section{$\cP$-triviality of suspension of the stunted real projective space}\label{4th}
We prove Theorem \ref{main1} in this section. We break the proof into several smaller lemmas. The $\cP$-triviality of $X_m=X_{m, 0}^0$ is already described in Corollary \ref{triviality_of_RPn}. We begin with $X^k_m=X_{m,0}^k$ where $k>0$.
\begin{lemma}\label{triviality_of_suspension_of_RPm}
$X_m^k=\Sigma^k \mathbb{RP}^m$ is not $\mathcal P$-trivial if and only if $m$ is odd and $(m+k)\equiv 0\pmod 4$.
\end{lemma}
\begin{proof}
Let $k\neq 2, 4$. If either $m$ or $k$ is even, then $X_m^k$ is known to be $\cc$-trivial  (see \cite[Theorem 1.1]{NT2014}), and hence it is $\mathcal P$-trivial. If $k = 2$ or $4$ and $m<k$, then again $X_m^k$ is $\cc$-trivial, and hence $\mathcal P$-trivial. Therefore, we check the cases when both $k$ and $m$ are odd, and $k=2,4$ and $k\geq m$. We recall the cohomology groups of $X_m^k$ from (\ref{suspension_stunted_cohomology}).\\
Let $m$ and $k$ both be odd. Then $H^{2j}(X_m^k; \mathbb Z)=0$ for $0<2j < m+k$. Therefore, the non-trivial Pontrjagin classes can only be $(m+k)$-dimensional classes. Hence, if $m+k\not\equiv 0\pmod 4$, $X_m^k$ must be $\mathcal P$-trivial.\\
Let $k=2$ and $m\geq 2$. Let $m$ be even. Then $H^j(X_m^k; \mathbb Z) = 0$ for $0<j<m+2, j$ odd and we have the exact sequence
\begin{equation}\label{ses}
    0\to H^{j}(X^k_m; \bb Z_2)\to H^{j+1}(X^k_m; \bb Z)\to H^{j+1}(X^k_m; \bb Z)\xrightarrow{\rho_2}H^{j+1}(X^k_m; \bb Z_2)\to 0.
\end{equation}
Since all the groups are isomorphic to $\bb Z_2$, we conclude that $\rho_2$ is isomorphism between all even dimensions cohomology groups.\\
For $m$ odd, $H^j(X_m^k; \mathbb Z) = 0$ for $0<j<m+1, j$ odd, and we have the same exact sequence as in \eqref{ses}, for $0<j<m+1$, and $\rho_2$ is isomorphism between all even dimensions cohomology groups, except, possibly, on dimension $m+1$, where we have 
\[
0\to H^{m}(X^k_m; \bb Z_2)\xrightarrow{\beta} H^{m+1}(X^k_m; \bb Z)\to H^{m+1}(X^k_m; \bb Z)\xrightarrow{\rho_2}H^{m+1}(X^k_m; \bb Z_2)\to \bb Z\to \cdots.
\]
Since all the first four groups are isomorphic to $\bb Z_2$, the map $\beta$ is an isomorphism and, therefore, $\rho_2$ is injective.\\
Since $\rho_2$ is injective whenever $H^{2j}(X^k_m; \mathbb Z)$ is non-trivial, irrespective of the parity of $m$, $X_m^2$ is $\mathcal P$-trivial by Corollary \ref{rho2_injective}.\\
For $k=4$, the same argument as $k=2$ shows that $X_m^4$ is $\mathcal P$-trivial, for $m\geq 4$.\\
It remains to show that $X_m^k$ is not $\cP$-trivial if $m, k$ odd, and $(m+k)\equiv 0\pmod 4$. Recall that $H^{m+k}(X_m^k; \mathbb Z)=\mathbb Z$, and let $\Sigma^kf_{m}\colon X_m^k\to S^{m+k}$ be a degree one map inducing isomorphism between the $(m+k)$-dimensional cohomology groups. (Note that such a map $f_m\colon X_m\to S^m$, that induces $\Sigma^kf_m$, exists as $X_m$ is orientable.) Let $\eta$ be a (complex) vector bundle over $S^{m+k}$ such  that $c_\frac{m+k}{2}(\eta)\neq 0$, in fact, $c_\frac{m+k}{2}(\eta)$ is an even multiple of the generator of $H^{m+k}(S^{m+k}; \mathbb Z)$. Let $\xi = f_{m, k}^*(\eta)_{\mathbb R}$. Since, by naturality of Chern classes, $c_\frac{m+k}{2}(f^*(\eta))\neq 0$ , from (\ref{chern_pontrjagin}), we get $p_{\frac{m+k}{42}}(\xi)=\pm 2c_\frac{m+k}{2}(\xi\otimes \mathbb C)\neq 0$. Therefore, $ X^k_m$ is not $\cP$-trivial. This completes the proof.
\end{proof}
Let us now discuss the $\cP$-triviality of $X_{m, n}$ where $n\neq 0$. To begin with, we have
\begin{lemma}
 $X_{m,1}$ is not $\cP$-trivial for $m\ge 4$.   
\end{lemma}
\begin{proof}
Consider the long exact sequence 
\[
\cdots \rightarrow H^2(X_{m,1}; \bb Z)\xrightarrow{q^*} H^2(\mathbb{RP}^m; \bb Z)\rightarrow H^2(\mathbb{RP}^1; \bb Z)\rightarrow\cdots,
\]
corresponding to the co-fiber sequence 
\[
\mathbb{RP}^1\xhookrightarrow{\iota} \mathbb{RP}^m \xrightarrow{q} X_{m,1}.
\]
Since $H^2(\mathbb{RP}^1; \bb Z)=0, q^*$ is surjective. Let $x\in H^2(X_{m,1}; \bb Z)\cong \mathbb{Z}_2$ be the generator. Then $q^*x$ is a generator of $H^2(\mathbb{RP}^m; \bb Z)\cong\mathbb{Z}_2$. Therefore, $x\smile x\neq 0$ as $q^*(x\smile x)=q^*x\smile q^*x\neq 0$. Hence, by Lemma \ref{odd1},  $X_{m,1}$ is not $\cP$-trivial.
\end{proof}
It is clear from the cell structure of $\bb{RP}^m$ that $\bb{RP}^m/\bb{RP}^{m-1}\simeq S^m$ and,
\begin{equation}\label{RPm/RPm-2 identification}
    \bb {RP}^{m}/\bb {RP}^{m-2}\simeq\begin{cases}
    \Sigma^{m-2}\bb{RP}^2; & \textit{if $m$ is even,}\\
    S^{m}\vee S^{m-1}; & \textit{if $m$ is odd.}
\end{cases}
\end{equation}

Therefore, from Lemma \ref{spheres} and Lemma \ref{triviality_of_suspension_of_RPm}, we have,
\begin{lemma}\label{X_m, m-1 and X_m, m-2}
    $X_{m, m-1}$ is not $\cP$-trivial if and only if $m\equiv 0\pmod 4$. $X_{m, m-2}$ is not $\cP$-trivial if and only if $m\equiv 1\pmod 4$.
\end{lemma}
Having settled Lemma \ref{X_m, m-1 and X_m, m-2}, we may assume $m-n>2$ and analyze the $\cP$-triviality of $X_{m, n}$. Let us first prove the following lemma.
\begin{lemma}\label{surjectivity_n_cong_3_mod_4}
    Let $n\equiv 3\pmod 4$. The homomorphism  $(\iota^\prime)^*\colon \widetilde{K}(X_{m, n})\to \widetilde{K}(X_{n+2, n})$ induced by $\iota^\prime\colon X_{n+2, n}\to X_{m, n}$ is surjective. 
\end{lemma}
\begin{proof}
    Let $\ga_p, \eta_p$ be the canonical line bundles over $X_p$ and $Y_p$ respectively, and $\lambda_p=\ga_p-1$ and $\mu_p=\eta_p-1$ be the generators of $\widetilde{KO}(X_p)$ and $\widetilde{K}(Y_p)$. Then, $\mathfrak c(\lambda_{2p+\epsilon})=\pi^*(\mu_p)$, where $\pi\colon  X_{2p+1}\to Y_p$ is the canonical projection. We denote, as Adams, $\nu_p=\mathfrak \pi^*(\mu_p)$.\\
    
    Let us write $n=4t+3$ and assume $m=2q+1$. Recall from \cite[Theorem 7.3]{Adams1962} that, $\widetilde{K}(X_{m, 4t+3})\cong \bb Z + \widetilde{K}(X_{m,4t})$ and $\widetilde{K}(X_{4t+5, 4t+3})\cong \bb Z + \widetilde{K}(X_{4t+5,4t})\cong \bb Z$. The $\bb Z$ summand in $\widetilde{K}(X_{m, 4t+3})$ and $\widetilde{K}(X_{4t+5, 4t+3})$ are generated by $\bar{\nu}_q^{(2t+2)}$ and $\bar{\nu}_{2t+2}^{(2t+2)}$, the elements corresponding to $\nu_q^{2t+2}\in \widetilde{K}(X_m)$ and $\nu_{2t+2}^{2t+2}\in \widetilde{K}(X_{2t+2})$, respectively. Consider the following commutative diagram.
    \[\begin{tikzcd}[scale cd=1]
         \widetilde{K}(X_{2q, 4t+3})\arrow[d,  "(\iota^\prime)^*"]&\widetilde{K}(X_{2q+1, 4t+3})\arrow[l, ""]\arrow[d, "(\iota^\prime)^*"]&\widetilde{K}(Y_{q, 2t+1})\arrow[d, "(\iota_\bb C^\prime)^*"]\arrow[l, "(\pi^\prime)^*"]\arrow[r, ""]& \widetilde{K}(Y_{q})\arrow[d, "\iota_\bb C^*"]\\
         \widetilde{K}(X_{4t+5, 4t+3})& \widetilde{K}(X_{4t+5, 4t+3})\arrow[l, "="]&\widetilde{K}(Y_{2t+2, 2t+1})\arrow[l, "(\pi^\prime)^*"]\arrow[r, ""]&\widetilde{K}(Y_{2t+2})
    \end{tikzcd}\]
    If $\iota_\bb C$ be the inclusion from $Y_{2t+2}$ to $Y_q$ then $\iota_\bb C^*(\mu_{q})=\mu_{2t+2}$, and hence $(\iota^\prime)^*(\bar{\nu}_q^{(2t+2)})=\bar{\nu}_{2t+2}^{(2t+2)}$. This proves that $(\iota^\prime)^*$ is surjective if $m$ is odd. \\

    If $m=2q$, then $\widetilde{KO}(X_{m, 4t+3})$ is generated by $\overline{(i^\prime)^*(\nu_q)}^{(2t+2)}$, where $i^\prime\colon X_{m, n}\to X_{m+1,n}$ is induced by the inclusion map $X_{m, n}\hookrightarrow X_{m+1,n}$. Hence, the proof follows from the commutativity of the leftmost square.    
\end{proof}
\begin{lemma}\label{X_m,n_n_cong_3_mod_4}
$X_{m, n}$ is not $\cP$-trivial for $n\equiv 3\pmod 4$.
\end{lemma}
\begin{proof}

The homomorphism $(\iota^\prime)^*\colon \widetilde{K}(X_{m, n})\to \widetilde{K}(X_{n+2, n})$ is surjective by Lemma \ref{surjectivity_n_cong_3_mod_4}. Replacing $m$ by $n$ in (\ref{RPm/RPm-2 identification}) and Lemma \ref{X_m, m-1 and X_m, m-2}, we have $\bb {RP}^{n+2}/\bb {RP}^n$ is not $\cP$-trivial as $n\equiv3\pmod 4$. 
Therefore, by Lemma \ref{comparison_lemma}, $X_{m, n}$ is not $\cP$-trivial for $n\equiv 3\pmod 4$.
\end{proof}
Let us recall the function $\varphi$ from Section \ref{intro}.
\begin{lemma}\label{P_triviality_necessary}
    Let $g=\varphi(n, 0)$. If $n\not\equiv 3\pmod 4$, then $X_{m, n}$ is $\cP$-trivial only if $m<2^{g+1}$.
\end{lemma}
\begin{proof}
    The map induced by the quotient map $q\colon\bb{RP}^m\rightarrow X_{m,n}$ maps $\widetilde{KO}(X_{m, n})$ isomorphically onto the subgroup of $\widetilde{KO}(\bb{RP}^m)$ generated by $2^{g}\lambda$, where $\lambda=\gamma_m -1$ and $\gamma_m$ is the tautological line bundle over $\bb{RP}^m$ (see \cite[Theorem 7.4]{Adams1962}). Let $\eta$ be a generator of $\widetilde{KO}(X_{m, n})$ such that $q^*(\eta) = 2^g\lambda$. Let $x$ be the unique non-zero element of $H^1(\bb{RP}^m; \bb Z_2)$. Then,
    \begin{equation}\label{total_Whitney_class_of_lambda}
        w(2^g\lambda)=(1+x)^{2^g}=1+x^{2^g}
    \end{equation}
    Note that $H^j(\bb{RP}^n; \bb Z_2) = 0$ for $j\geq n+1$. Hence, $q^*\colon H^j(X_{m, n}; \bb Z_2)\to H^j(\bb{RP}^m; \bb Z_2)$ is an isomorphism for $j\ge n+1$. Since we have $w_{2^g}^2(\eta)=q^*(x^{2^{g+1}})\neq 0$ for $m\ge 2^{g+1}$, by
    Lemma \ref{square_of_Whitney_class}, $X_{m, n}$ cannot be $\cP$-trivial if $m\ge 2^{g+1}$. Therefore, if $X_{m,n}$ is $\cP$-trivial, then $m< 2^{g+1}$.
\end{proof}

\begin{corollary}\label{X_m,n_0_mod_2}
    If $n\equiv 0\pmod2$, then $X_{m, n}$ is $\cP$-trivial if and only if $m<2^{g+1}$.
\end{corollary}
\begin{proof}
    Given Lemma \ref{P_triviality_necessary}, it is enough to prove that $m<2^{g+1}$ is sufficient for $\cP$-triviality. Recall that if $n=2t$, then $X_{m, n}$ is $\cc$-trivial if and only if $m<2^{t+1}$ (see \cite[Theorem 1.2(3)]{NT2014}). Since $n$ is even, by definition of $\varphi(n, 0)$, we have $g=t$. Therefore, if $m<2^{g+1}$, then $X_{m, n}$ is $\cc$-trivial and hence $\cP$-trivial.
\end{proof}
\begin{corollary}\label{X_m,n_5_mod_8}
    If $n\equiv 5\pmod 8$, then $X_{m, n}$ is $\cP$-trivial if and only if $m<2^{g+1}$.
\end{corollary}
\begin{proof}
     Again, it is enough to show that $m<2^{g+1}$ is sufficient for $\cP$-triviality. It is obvious from the definition of $\varphi$ that $g^\prime=\varphi(n-1, 0)=\varphi(n, 0)=g$ when $n\equiv 5\pmod 8$. Now $n-1\equiv 0\pmod 4$, and if $m<2^{g+1}$ then $m<2^{g^\prime+1}$. Therefore, $X_{m, n-1}$ is $\cP$-trivial by Corollary \ref{X_m,n_0_mod_2}.

    Now, if possible, let $\xi$ be a bundle over $X_{m, n}$ with $p_i(\xi)\neq 0$. Since $n\equiv 5\pmod8$, we must have $4i\geq n+3$. Recall that the obvious map $q\colon X_{m, n-1}\to X_{m, n}$ induces isomorphisms $q^*\colon H^j(X_{m, n}; \bb Z)\to H^j(X_{m, n-1}; \bb Z)$ for all $j>n+1$. Then $q^*(p_i(\xi))\neq 0$ as $q^*\colon H^{4i}(X_{m,n}; \bb Z)\to H^{4i}(X_{m,n-1}; \bb Z)$ is an isomorphism. Therefore, $p_i(q^*(\xi))\neq 0$, which contradicts that $X_{m, n-1}$ is $\cP$-trivial. Therefore, there cannot be a vector bundle over $X_{m, n}$ with non-trivial total Pontrjagin class, and hence, $X_{m, n}$ is $\cP$-trivial.
\end{proof}
\begin{corollary}\label{X_m,n_1_mod_8}
    If $n\equiv 1\pmod 8$, then $X_{m, n}$ is $\cP$-trivial if and only if $m<2^{g+1}$.
\end{corollary}
\begin{proof}
    For $n\equiv 1\pmod 8$, we have $g^\prime=\varphi(n-1, 0)= \varphi(n, 0)-1=g-1.$ Therefore, if $m<2^{g^\prime+1}=2^g$, the same argument as Corollary \ref{X_m,n_5_mod_8}  shows that $X_{m, n}$ is $\cP$-trivial.

    Let $2^g\le m<2^{g+1}$. It is easy to see that the reduction modulo two homomorphism, $\rho_2\colon H^{2j}(X_{m, n}; \bb Z)\to H^{2j}(X_{m, n}; \bb Z_2)$ is injective whenever $H^{2j}(X_{m, n}; \bb Z)\neq 0$. If $X_{m, n}$ is $\mathcal W$-trivial, then by \eqref{Pontrjagin_Whitney}, $X_{m, n}$ is $\cP$-trivial. If, however, $X_{m, n}$ is not $\mathcal W$-trivial, since $\alpha$ is a generator of $\widetilde{KO}(X_{m, n})$, from \eqref{total_Whitney_class_of_lambda}, the Stiefel-Whitney classes of dimension less than $2^g$ vanishes for any vector bundle over $X_{m, n}$. Therefore, as $m<2^{g+1}$, the square of any nontrivial Stiefel-Whitney class is zero, and we can again apply \eqref{Pontrjagin_Whitney} to conclude that $X_{m, n}$ is $\cP$-trivial.
    
\end{proof}

To conclude, let $0< n<m$ and $k> 0$. We have from \cite[Theorem 1.2]{NT2014} that $X_{m, n}^k$ is not $C$-trivial if both $k, m$ are odd or $k$ even $n$ odd. Therefore, as $\cc$-triviality implies $\cP$-triviality, we only look into these subfamily of spaces, that is, $X^k_{m, n}$ where both $k, m$ are odd or $k$ even $n$ odd, while discussing non-$\cP$-triviality of $X^k_{m, n}$.



\begin{lemma}\label{p-triviality_of_suspension}
$X^k_{m, n}$ is not $\mathcal P$-trivial if and only if one of the following hold.\\
(a) $k$ and $m$ both are odd and $(k+m)\equiv 0\pmod 4$.\\
(b) $k$ even, $n$ odd and $k+n+1\equiv 0, 4\pmod 4.$
\end{lemma}
\begin{proof}
$(a)$ If $k$ and $m$ both are odd and $k+m\equiv 2\pmod 4$, since there are no non-trivial even dimensional cohomology groups, $X^k_{m, n}$ is $\mathcal P$-trivial. Therefore, we are left with the case $k, m$ odd and $k+m\equiv 0\pmod 4$. Since there is a degree one map from $X^k_{m, n}\to S^{k+m}$ inducing isomorphism on the top cohomology group, and, since $S^{k+m}$ is not $\mathcal P$-trivial for $k+m\equiv 0\pmod 4$, $X^k_{m, n}$ also not $\mathcal P$-trivial. This proves $(a)$.\\

$(b)$ Now, let $k$ be even and $n$ be odd. Furthermore, assume that $k+n+1\equiv 0\pmod 4$ and $m$ is even. There is a complex bundle $\xi$ over $X^k_{m, n}$ such that $c_{\frac{k+n+1}{2}}(\xi)\neq 0$ (see \cite[Theorem 1.2(5)]{NT2014}). Note, here $n$ is odd, and hence for $j=n+k-1, j-k=n+1$ is even. Therefore, from (\ref{suspension_stunted_cohomology}), we have, $H^{n+k+1}(X_{m, n}^k; \bb Z)\cong \bb Z$. From (\ref{chern_pontrjagin}), we have $p_{\frac{k+n+1}{4}}(\xi_{\mathbb R})=\pm 2c_{\frac{k+n+1}{2}}(\xi)\neq 0$ and hence $X^k_{m, n}$ is not $\mathcal P$-trivial.

If, however, $m$ is odd, consider the isomorphism $(\Sigma^k\iota^\prime)^*\colon H^{n+k+1}(X_{m+1, n}^k; \bb Z)\to H^{n+k+1}(X_{m, n}^k; \bb Z)$ induced by $\iota^\prime\colon X_{m, n}\to X_{m+1, n}$. Since $m+1$ is even and $k+n+1\equiv 0\pmod 4$, by the above case $X^k_{m+1, n}$ is not $\cP$-trivial. Let $\xi$ be a complex bundle $\xi$ over $X^k_{m, n}$ such that $c_{\frac{k+n+1}{2}}(\xi)\neq 0$. Then $p_{\frac{n+k+1}{4}}((\Sigma^k\iota^\prime)^*(\xi)_\bb R)\neq 0$, and therefore, $X^k_{m, n}$ is not $\cP$-trivial.\\

Let $k$ be even and $k+n+1\equiv 2\pmod 8$. By Corollary \ref{rho2_injective}, it is enough to show that $\rho_2\colon H^{4j}(X^k_{m, n}; \bb Z)\to H^{4j}(X^k_{m, n}; \bb Z_2)$ is injective whenever $H^{4j}(X^k_{m, n}; \bb Z)$ is non-trivial. Note that, as $n+k+1\equiv 2\pmod 4$, the  $4j$-dimensional cohomology groups vanish for dimension less than $n+k+2$. Therefore, $H^{4j}(X^k_{m, n}; \bb Z)\neq 0$ for $m\ge 4j\geq n+k+3$, and form (\ref{suspension_stunted_cohomology}), we have the following exact sequence.
\[
0\to H^{4j}(X^k_{m, n}; \bb Z_2)\xrightarrow{\beta} H^{4j}(X^k_{m, n}; \bb Z)\to H^{4j}(X^k_{m, n}; \bb Z)\xrightarrow{\rho_2}H^{4j}(X^k_{m, n}; \bb Z_2)\to\cdots
\]
Since all the groups are isomorphic to $\bb Z_2$, $\rho_2$ is an isomorphism. Hence $X_{m,n}$ is $\cP$-trivial in this case.\\

Finally, let $k$ is even and $k+n+1\equiv 6\pmod 8$. Although a similar set of arguments as in the previous case shows that $\rho_2$ is an isomorphism in the required dimensions, we give a different proof of non-$\cP$-triviality. Let us consider the long exact sequence
\[
\cdots\to \widetilde{KO}^{-k}(X_{m, n+1})\to \widetilde{KO}^{-k}(X_{m, n})\to \widetilde{KO}^{-k}(S^{n+1})\to\cdots,
\]
corresponding to the co-fiber sequence $X_{n+1, n}\to X_{m, n}\to X_{m, n+1}$. Since $\widetilde{KO}^{-k}(S^{n+1})=0$ and, as $n+1$ is even, $X_{m, n+1}$ is $\cP$-trivial, $X_{m, n}$ cannot but be $\cP$-trivial. This completes the proof.
\end{proof}

Theorem \ref{main1} now follows from Corollary \ref{triviality_of_RPn}, Lemma \ref{triviality_of_suspension_of_RPm}, \ref{X_m,n_n_cong_3_mod_4}, \ref{p-triviality_of_suspension} and Corollary \ref{X_m,n_0_mod_2}, \ref{X_m,n_5_mod_8}, \ref{X_m,n_1_mod_8}.
\section{$\cP$-triviality of suspension of the stunted complex projective space}\label{5th}


We prove Theorem \ref{main2} in this section. To begin with, we have the following crucial lemma.
\begin{lemma}\label{suspension_of_CP2}
    $Y^k_2$ is $\cP$-trivial if and only if $k$ is odd.
\end{lemma}
\begin{proof}
Since there is no non-trivial even dimensional cohomology group of $\Sigma^k\bb {CP}_2$ for $k$ odd, $Y^k_2$ is certainly $\cP$-trivial for $k$ odd. We show that $Y^k_2$ cannot be $\cP$-trivial for $k$ even.

    Assume, first, that $k\equiv 0,4 \pmod 8$, $i.e.,$ $k\equiv 0\pmod 4$. Then $Y^k_2$ is not $\cP$-trivial by Proposition \ref{suspension_of_orientable_manifolds}.\\


If $k\equiv 2\pmod 8$, the co-fiber sequence $S^2 \xrightarrow{\iota} \mathbb{CP}^2 \xrightarrow{q} S^4$ gives rise to the following long exact sequence.
\[
 \cdots \rightarrow\widetilde{KO}^{-k}(S^4)\xrightarrow{q^*} \widetilde{KO}^{-k}(\mathbb{CP}^2) \xrightarrow{\iota^*} \widetilde{KO}^{-k}(S^2)\xrightarrow{\delta} \widetilde{KO}^{-k+1}(S^4)\xrightarrow{q^*} \widetilde{KO}^{-k+1}(\mathbb{CP}^2) \rightarrow \cdots
\]
Since $k+3\equiv 5 \pmod 8$, we have $\widetilde{K O}^{-k+1}\left(S^{4}\right)=\widetilde{KO}(S^{k+3})=0$, and the above exact sequence gives rise to the exact sequence 
\[
\cdots \rightarrow\widetilde{KO}^{-k}(S^4)\xrightarrow{q^*} \widetilde{KO}(\Sigma^k\mathbb{CP}^2) \xrightarrow{\iota^*} \widetilde{KO}(S^{k+2})\rightarrow 0.
\]
 Now, $S^{k+2}$ is not $\cP$-trivial by Lemma \ref{spheres}, as $k+2\equiv 0\pmod 4$. Since  the map $\iota^*$ is surjective, $Y^k_2$ cannot be $\cP$-trivial by Lemma \ref{comparison_lemma}.\\

To conclude, let $k \equiv 6 \pmod 8$. We have the following long exact sequence of complex K-theory.
\[
\cdots \rightarrow\widetilde{K}^{-k}(\mathbb{CP}^2)\xrightarrow{\iota^*} \widetilde{K}^{-k}(S^2) \xrightarrow{\delta} \widetilde{K}^{-k+1}(S^4)\rightarrow\cdots.
\] 
Since $(-k+3)\equiv 1\pmod 2, \widetilde{K}^{-k+1}(S^4) = \widetilde{K}^{-k+3}(S^4) =0$, the above exact sequence gives rise to the exact sequence 
\[
\cdots \rightarrow\widetilde{K}(\Sigma^k\mathbb{CP}^2)\xrightarrow{\iota^*} \widetilde{K}(S^{k+2}) \rightarrow 0.
\]
Now $k+2\equiv 0\pmod 8$, and let $k+2=4t$. By the discussion before Lemma \ref{spheres}, there exists a complex vector bundle $\xi$ over $S^{k+2}$ such that $c_{2t}(\xi)\neq 0$. Since $\iota^*$ is surjective, there exist a bundle $\eta$ over $Y^k_2$ such that $\iota^*(\eta)=\xi$. Since $\iota^*(c_{2t}(\eta))=c_{2t}(\xi)\neq 0$, hence, $0\neq c_{2t}(\eta)=p_t(\eta_\bb R)$. Hence, $Y^k_2$ is not $\cP$-trivial again. This completes the proof.
\end{proof}



We are now ready to prove Theorem \ref{main2}.
\begin{lemma}\label{main_lemma_main2}
Let $k\geq 0$ be even. The stunted complex projective space $Y^k_{m, m-1}$ is not $\cP$-trivial if and only if $(2m+k)\equiv 0\pmod 4$. And, $Y^k_{m, n}$ is never $\cP$-trivial for $n\leq m-2$. 
\end{lemma}
\begin{proof}
Note that for $n=m-1, Y^k_{m, n}=S^{2m+k}$, and the statement follows from Lemma \ref{spheres}.\\

Let us assume $k=0$. Further, let $n = m-2$. Recall that 
\[
Y_{m, m-2}=\begin{cases}
S^{2m}\vee S^{2m-2} \textit{, for $m$ odd;}\\
\Sigma^{2m-4} \bb{CP}^2 \textit{, for $m$ even.}
\end{cases}
\]
Now, by the Lemma \ref{suspension_of_CP2}, as $\Sigma^{2m-4} \bb{CP}^2$ is never $\cP$-trivial for $m$ even, $Y_{m, n}$ is not $\cP$-trivial for $m$ even. For $m$ odd, $(2m-2)\equiv 0\pmod 4$, and $S^{2m-2}$ is not $\cP$-trivial. Therefore, $Y_{m, m-2}$ is not $\cP$ trivial for $m$ odd. In fact, the pullback of a bundle with non-trivial total Pontrjagin class over $S^{2m-2}$ by the projection $Y_{m, m-2}\to S^{2m-2}$ has non-trivial total Pontrjagin class.\\

For $n\leq m-3$, the sequence of inclusions $\bb{CP}^n\hookrightarrow \bb{CP}^{n+2}\hookrightarrow \bb{CP}^m$ induces the co-fiber sequence
 \[
 Y_{n+2, n}\to Y_{m, n}\to Y_{m, n+2}.
 \]
Since the induced map $H^{j}(Y_{m, n}; \bb Z)\to H^{j}(Y_{n+2,n}; \bb Z)$ is an isomorphism on dimensions $2n+4$ and $2n+2$, the non $\cP$-triviality of $Y_{m, n}$ follows from that of $Y_{n+2, n}$ by considering the pullback of a relevant bundle as in the previous case. \\

Now, let us assume $k\geq 2$. Then, $\Sigma^k Y_{m,m-2}$ is either a wedge of two spheres, at least one of which is of dimension multiple of four, or an even suspension of $\bb{CP}^2$ and, therefore, never $\cP$-trivial by Lemma \ref{spheres} and \ref{suspension_of_CP2}.\\

For $n\leq m-3$, the cellular inclusion $\Sigma^k Y_{n+2, n}\to \Sigma^k Y_{m, n}$ induces, as in the case when $k=0$, an isomorphism $H^{j}(\Sigma^k Y_{m,n}; \bb Z)\to H^{j}(\Sigma^k Y_{n+2,n}; \bb Z)$ in dimensions $j= k+2n+4, k+2n+2$, the non $\cP$-triviality of $\Sigma^k Y_{m, n}$ follows form that of $\Sigma^k Y_{n+2, n}$.
\end{proof}

Since $Y_{m,n}$ has nontrivial cohomology groups in even dimensions only, if $k$ is odd, $Y^k_{m,n}$ has nontrivial cohomology groups in odd dimensions only, and, hence $Y^k_{m,n}$ is $\cP$-trivial in this case. This, in addition to Lemma \ref{suspension_of_CP2} and \ref{main_lemma_main2}, completes the proof of Theorem \ref{main2}.

{\bf Acknowledgments:} The authors would like to thank Aniruddha Naolekar for introducing the problem and the productive discussion they had with him.

{\bf Funding:} The first named author was partially funded by the \address{\href{http://dx.doi.org/10.13039/501100005276}{National Board for Higher Mathematics} during the project. The second named author was funded by the \address{\href{http://dx.doi.org/10.13039/501100013357}{Indian Statistical Institute}, Bangalore Centre. 

\bibliography{main}

\begin{thebibliography}{10}

\bibitem{Adams1962}
J.~F. Adams.
\newblock Vector fields on spheres.
\newblock {\em Annals of Mathematics}, 75(3):603--632, 1962.

\bibitem{AH1961}
M.~F. Atiyah and F.~Hirzebruch.
\newblock Bott periodicity and the parallelizability of the spheres.
\newblock {\em Proc. Cambridge Philos. Soc.}, 57:223--226, 1961.

\bibitem{BKN2024}
A.~C. Bhattacharya, B.~Kundu, and A.~C. Naolekar.
\newblock $w$-triviality of low dimensional manifolds, 2024.

\bibitem{Fujii1967}
M.~Fujii.
\newblock {$K_{O}$}-groups of projective spaces.
\newblock {\em Osaka Math. J.}, 4:141--149, 1967.

\bibitem{hu}
D.~Husem{\"o}ller.
\newblock {\em Fibre bundles}, volume~5.
\newblock Springer, 1966.

\bibitem{James1959}
I.~James.
\newblock Spaces associated with stiefel manifolds.
\newblock {\em Proceedings of the London Mathematical Society}, 3(1):115--140, 1959.

\bibitem{Milnor1958}
J.~Milnor.
\newblock Some consequences of a theorem of {B}ott.
\newblock {\em Ann. of Math. (2)}, 68:444--449, 1958.

\bibitem{milnor1974}
J.~W. Milnor and J.~D. Stasheff.
\newblock {\em Characteristic classes}.
\newblock Number~76. Princeton university press, 1974.

\bibitem{NT2014}
A.~C. Naolekar and A.~S. Thakur.
\newblock On trivialities of {C}hern classes.
\newblock {\em Acta Math. Hungar.}, 144(1):99--109, 2014.

\bibitem{NT2014(1)}
A.~C. Naolekar and A.~S. Thakur.
\newblock Vector bundles over iterated suspensions of stunted real projective spaces.
\newblock {\em Acta Math. Hungar.}, 142(2):339--347, 2014.

\bibitem{Tanaka2010}
R.~Tanaka.
\newblock On trivialities of {S}tiefel-{W}hitney classes of vector bundles over iterated suspension spaces.
\newblock {\em Homology Homotopy Appl.}, 12(1):357--366, 2010.

\bibitem{cadek1993}
M.~\v{C}adek and J.~Van\v{z}ura.
\newblock On the classification of oriented vector bundles over {$5$}-complexes.
\newblock {\em Czechoslovak Math. J.}, 43(118)(4):753--764, 1993.

\end{thebibliography}
\bibliographystyle{abbrv}
\end{document}